\documentclass{amsart}
\usepackage{amsmath,amssymb,amsthm,mathtools}
\usepackage{enumitem}
\usepackage{tikz-cd, hyperref}

\newtheorem{theorem}{Theorem}[section]

\newtheorem{proposition}[theorem]{Proposition}
\newtheorem{corollary}[theorem]{Corollary}
\numberwithin{equation}{section}
\theoremstyle{definition}

\newtheorem{remark}[theorem]{Remark}

\newcommand{\Z}{\mathbb Z}
\newcommand{\R}{\mathbb R}
\newcommand{\C}{\mathbb C}
\renewcommand{\P}{\mathbb P}
\newcommand{\cO}{\mathcal O}
\newcommand{\cR}{\mathcal R}
\newcommand{\CP}{\mathbb{CP}}
\newcommand{\wideand}{ \quad \text{ and } \quad }

\DeclareMathOperator{\cone}{cone}

\DeclareMathOperator{\SO}{SO}
\DeclareMathOperator{\PD}{PD}

\DeclareMathOperator{\id}{id}
\DeclareMathOperator{\Int}{Int}

\title{Cohomological rigidity of smooth toric Fano fourfolds}

\author[S. Choi]{Suyoung Choi}
\address{Department of Mathematics, Ajou University, 206, World Cup-ro, Yeongtong-gu, Suwon 16499, Republic of Korea}
\email{schoi@ajou.ac.kr}

\date{\today}
\thanks{This work was supported by the National Research Foundation of Korea Grant funded by the Korean Government (RS-2025-00521982).}
\subjclass[2020]{Primary 14M25; Secondary 57S12, 57R19, 14J45}
\keywords{cohomological rigidity, smooth toric Fano fourfolds, Bott manifolds, bistellar moves, toric flips}

\begin{document}
\begin{abstract}
    We prove that any two smooth toric Fano fourfolds whose integral cohomology rings are isomorphic as graded rings are diffeomorphic.
\end{abstract}
\maketitle

\section{Introduction}

The integral cohomology ring is not a complete invariant of closed smooth manifolds in general. 
For a geometrically restricted class, however, it can be a surprisingly strong invariant. 
This leads to the \emph{cohomological rigidity problem}: for a given class of manifolds, does an isomorphism of their integral cohomology rings as graded rings imply that the manifolds are diffeomorphic?
In this paper, a \emph{toric manifold} means a smooth complete toric variety.
The cohomological rigidity problem for toric manifolds was proposed by Masuda and Suh~\cite{Masuda-Suh2008} and has since been studied for many natural subclasses.
The defining complete nonsingular fan encodes the toric variety and also gives an explicit presentation of its integral cohomology ring.
Ordinary cohomology, however, forgets the torus action and need not retain the full fan data.
It is therefore not clear a priori whether it determines the underlying smooth manifold.

Among toric manifolds, smooth toric Fano varieties form a particularly important and manageable class. 
A smooth projective variety is \emph{Fano} if its anticanonical bundle is ample, and Fano varieties form one of the basic classes in algebraic geometry. 
In the toric setting, smooth toric Fano varieties correspond to smooth Fano polytopes. 
Consequently, there are only finitely many of them in each fixed dimension, and their classification becomes a
finite combinatorial problem. 
In complex dimension four, the classification due to Batyrev and Sato consists of exactly $124$ smooth toric Fano varieties up to isomorphism \cite{Batyrev1999,Sato2000}. 
This explicit algebraic classification makes it natural to ask for the coarser classification of their underlying smooth manifolds. 
Indeed, toric varieties defined by different fans may nevertheless be diffeomorphic.

Higashitani--Kurimoto--Masuda \cite{Higashitani-Kurimoto-Masuda2022} studied this topological classification through the cohomological rigidity problem. 
They proved that smooth toric Fano threefolds are cohomologically rigid and obtained the same conclusion for smooth toric Fano fourfolds except for a single pair. 
Let $X_{50}$ and $X_{57}$ denote the smooth toric Fano fourfolds with ID numbers $50$ and $57$ in their list. 
They proved that the integral cohomology rings of $X_{50}$ and $X_{57}$ are isomorphic and that one such isomorphism preserves the total Pontryagin class. 
Thus neither the cohomology ring nor the Pontryagin classes distinguish the two manifolds. 
They left open whether the two manifolds are diffeomorphic or even homeomorphic~\cite[Question~5.1]{Higashitani-Kurimoto-Masuda2022}. 
We settle this remaining case.

\begin{theorem} \label{thm:main}
    The toric manifolds $X_{50}$ and $X_{57}$ are orientation-preservingly diffeomorphic.
\end{theorem}

Together with the classification results of~\cite{Higashitani-Kurimoto-Masuda2022}, the theorem completes the cohomological rigidity problem for smooth toric Fano fourfolds.

\begin{corollary}
    Any two smooth toric Fano fourfolds are diffeomorphic if their integral cohomology rings are isomorphic as graded rings.
\end{corollary}

\section{The Bott manifolds}\label{sec:fan-bott}

After an appropriate relabeling of the rays and a unimodular change of lattice basis\footnote{Relative to the ray ordering in \cite[Table~27]{Higashitani-Kurimoto-Masuda2022}, the new orders are $(1,3,4,2,5,6,7,8)$ for $X_{50}$ and $(1,3,4,2,5,6,8,7)$ for $X_{57}$. The new lattice basis is $(e_1,e_3,e_4,e_2)$ in terms of the basis used there.}, we may write the fans of $X_{50}$ and $X_{57}$ as follows:
\begin{align*}
    v_1&=e_1,& v_2&=e_2,& v_3&=e_3,& v_4&=e_4,\\
    v_5&=-e_1+e_2+e_3-e_4,& v_6&=-e_2+e_4,& v_8&=-e_4,
\end{align*}
together with
$$
    v_7^{50}=-e_3+e_4 \wideand v_7^{57}=-e_3.
$$
We label each ray by the index of its primitive generator. 
For a subset $I=\{i_1,\ldots,i_k\}\subset \{1,\ldots,8\}$, we write $i_1\cdots i_k$ for $I$. 
Whenever $\sigma_I \coloneq \cone(v_j\mid j\in I)$ is a cone of the fan under consideration, let $V(I)$ denote the closure of the corresponding torus orbit. 
In particular, $D_j \coloneq V(j)$ is the torus-invariant prime divisor corresponding to the ray $\R_{\geq 0}v_j$.
With these conventions, the two fans have the same minimal nonfaces
$$
    26,\quad 37,\quad 48,\quad 145,\quad 156,\quad 157, \wideand 238.
$$
For each $i\in\{50,57\}$, let $\Sigma_i$ denote the fan of $X_i$, and $\Sigma_i^-$ the fan with primitive ray generators $v_1$, $\ldots$, $v_6$, $v_7^i$, and $v_8$ whose underlying simplicial complex is the cross-polytope complex with minimal nonfaces
\begin{equation}\label{eq:B-mnf}
    15,\quad 26,\quad 37, \wideand 48.
\end{equation}
The ray generators are in the standard triangular form of a four-stage Bott fan. 
Hence $\Sigma_i^-$ is smooth and complete. 
We denote the associated Bott manifold by~$B_i$.
For background on Bott manifolds, see, for example, \cite{Grossberg-Karshon1994, Masuda-Panov2008,Choi-Hwang-Jang2025}.

For a finite set $I$, let $\Delta_I$ denote the simplex with vertex set $I$. 
The underlying simplicial complexes of $\Sigma_i^-$ and $\Sigma_i$ differ only by the local bistellar move
$$
    \Delta_{238}*\partial\Delta_{15} \quad\longleftrightarrow\quad \partial\Delta_{238}*\Delta_{15}.
$$
Equivalently, at the level of simplices, this move replaces 
\begin{equation}\label{eq:bistellar}
    \{1238,2358\} \quad\longleftrightarrow\quad \{1235,1258,1358\}.
\end{equation}
This is the $2$--$3$ bistellar move shown in Figure~\ref{fig:bistellar}.
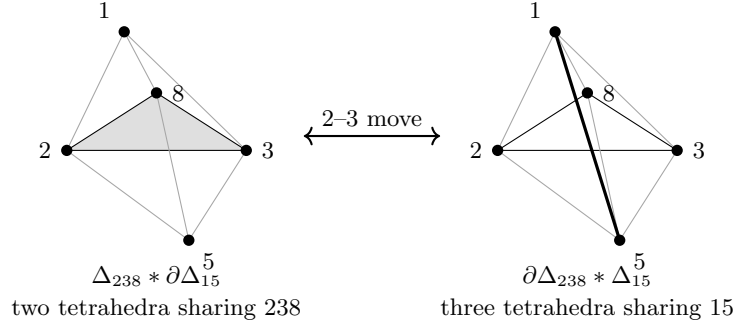
\begin{figure}[t]
    \centering
    \begin{tikzpicture}[
        scale=0.95,
        vertex/.style={circle,fill=black,inner sep=1.5pt},
        every node/.style={font=\small}
    ]
        \begin{scope}
            \coordinate (a2) at (-1.25,0);
            \coordinate (a3) at (1.25,0);
            \coordinate (a8) at (0,0.8);
            \coordinate (a1) at (-0.45,1.65);
            \coordinate (a5) at (0.45,-1.25);
        
            \fill[gray!25] (a2)--(a3)--(a8)--cycle;
            \draw (a2)--(a3)--(a8)--cycle;
            \foreach \x in {a2,a3,a8} {
                \draw[gray!70] (a1)--(\x);
                \draw[gray!70] (a5)--(\x);
            }
        
            \node[vertex,label=left:$2$] at (a2) {};
            \node[vertex,label=right:$3$] at (a3) {};
            \node[vertex,label=right:$8$] at (a8) {};
            \node[vertex,label=above left:$1$] at (a1) {};
            \node[vertex,label=below right:$5$] at (a5) {};
        
            \node at (0,-1.75) {$\Delta_{238}*\partial\Delta_{15}$};
            \node at (0,-2.2) {two tetrahedra sharing $238$};
        \end{scope}
        
        \draw[<->,thick] (2.05,0.2)--(3.95,0.2) node[midway,above] {$2$--$3$ move};
        
        \begin{scope}[xshift=6cm]
            \coordinate (b2) at (-1.25,0);
            \coordinate (b3) at (1.25,0);
            \coordinate (b8) at (0,0.8);
            \coordinate (b1) at (-0.45,1.65);
            \coordinate (b5) at (0.45,-1.25);
        
            \draw (b2)--(b3)--(b8)--cycle;
            \foreach \x in {b2,b3,b8} {
                \draw[gray!70] (b1)--(\x);
                \draw[gray!70] (b5)--(\x);
            }
            \draw[very thick] (b1)--(b5);
        
            \node[vertex,label=left:$2$] at (b2) {};
            \node[vertex,label=right:$3$] at (b3) {};
            \node[vertex,label=right:$8$] at (b8) {};
            \node[vertex,label=above left:$1$] at (b1) {};
            \node[vertex,label=below right:$5$] at (b5) {};
        
            \node at (0,-1.75) {$\partial\Delta_{238}*\Delta_{15}$};
            \node at (0,-2.2) {three tetrahedra sharing $15$};
        \end{scope}
    \end{tikzpicture}
    \caption{The local $2$--$3$ bistellar move. The shaded face on the left and the thick edge on the right lie in the interior of the triangular bipyramid.}
    \label{fig:bistellar}
\end{figure}
The corresponding primitive relation is
$$
    v_0 \coloneq v_1+v_5=v_2+v_3+v_8=(0,1,1,-1).
$$

\begin{proposition}\label{prop:common-flip}
    For each $i\in\{50,57\}$, the toric manifold $X_i$ is obtained from~$B_i$ by the standard flip corresponding to the above $2$--$3$ bistellar move.
    The centers in~$B_i$ and $X_i$ are $C_i=V(238)\cong\CP^1$ and $V(15)\cong\CP^2$, respectively.
    Moreover,
    $$
        \nu_{C_i/B_i}\cong\cO_{\CP^1}(-1)^{\oplus3}.
    $$
\end{proposition}
\begin{proof}
    Since $v_0=v_2+v_3+v_8=v_1+v_5$, the star subdivisions of $\Sigma_i^-$ at $238$ and of $\Sigma_i$ at $15$, both with new ray $\R_{\geq0}v_0$, give the same fan.
    Thus the corresponding modification from $B_i$ to $X_i$ is the standard flip, with centers $V(238)\cong\CP^1$ and $V(15)\cong\CP^2$ on the two sides.

    The wall relation associated with $C_i$ is $v_1+v_5-v_2-v_3-v_8=0$.
    Hence
    $$
        D_2\cdot C_i=D_3\cdot C_i=D_8\cdot C_i=-1.
    $$
    Since $C_i=D_2\cap D_3\cap D_8$ transversely, it follows that
    $$
        \nu_{C_i/B_i} \cong \cO(D_2)|_{C_i}\oplus \cO(D_3)|_{C_i}\oplus \cO(D_8)|_{C_i} \cong \cO_{\CP^1}(-1)^{\oplus3}.
    $$
\end{proof}

Set $x=[D_5]$, $y=[D_6]$, $z=[D_7]$, and $u=[D_8]$.
The standard presentation of the cohomology ring of a smooth toric manifold~with \eqref{eq:B-mnf} gives
\begin{align*}
    H^\ast(B_{50};\Z) &\cong \frac{\Z[x,y,z,u]}{(x^2, y(y-x), z(z-x), u(u+x-y-z))}, \wideand\\
    H^\ast(B_{57};\Z) &\cong \frac{\Z[x,y,z,u]}{(x^2, y(y-x), z(z-x), u(u+x-y))}.
\end{align*}
All four generators have degree two.
Let $\psi\colon H^\ast(B_{50};\Z)\to H^\ast(B_{57};\Z)$ be a homomorphism defined by 
$$
    \psi(x)=-x+2z, \quad \psi(y)=z, \quad \psi(z)=z-y, \wideand \psi(u)=-u.
$$
A direct check shows that this assignment defines a graded ring isomorphism.
In both rings, $xyzu$ is the positive generator of~$H^8(B_i;\Z)$ determined by the complex orientation. 
Since $\psi(xyzu)=(-x+2z)z(z-y)(-u)=xyzu$, $\psi$ preserves the complex orientation.

Since $\PD[C_i]=[D_2][D_3][D_8]=(y-x)(z-x)u$, we obtain
\begin{equation}\label{eq:center-sign}
    \psi\bigl(\PD[C_{50}]\bigr) =-\PD[C_{57}].
\end{equation}

Strong cohomological rigidity of Bott manifolds says that every integral graded cohomology-ring isomorphism between Bott manifolds is induced by a diffeomorphism~\cite[Theorem~1.1]{Choi-Hwang-Jang2025}. 
Therefore there is an orientation-preserving diffeomorphism~$\Phi_0\colon B_{57}\to B_{50}$ such that $\Phi_0^*=\psi$.
By~\eqref{eq:center-sign} and Poincar\'e duality, 
\begin{equation}\label{eq:homology-sign} 
    (\Phi_0)_*[C_{57}]=-[C_{50}]. 
\end{equation}

\section{Construction of a diffeomorphism} \label{sec:transport}

We first describe the local model of the flip. 
Let $V=\C^2$ and $W=\C^3$, and let~$\gamma_V$ and~$\gamma_W$ be the tautological line bundles over $\P(V)=\CP^1$ and $\P(W)=\CP^2$, respectively. 
Put
$$
    E_V\coloneq \gamma_V\otimes W \cong \cO_{\CP^1}(-1)^{\oplus3}  \wideand E_W \coloneq V\otimes\gamma_W \cong \cO_{\CP^2}(-1)^{\oplus 2}.
$$
Equip $V$ and $W$ with conjugation-invariant Hermitian metrics.
Let $U(V)$ and~$U(W)$ denote their groups of complex-linear isometries.
We use the induced norm on~$V\otimes W$.

The complements of the zero sections in $E_V$ and $E_W$ are canonically identified with the space of nonzero rank-one tensors in $V\otimes W$. 
Indeed, a nonzero rank-one tensor $\xi$ determines unique lines $\ell\subset V$ and $m\subset W$. 
This gives a canonical $U(V)\times U(W)$-equivariant diffeomorphism $\chi\colon E_V\setminus\P(V)\to E_W\setminus\P(W)$ defined by~$\chi(\ell,\xi)=(m,\xi)$.
With respect to the tensor-product norm, $\chi$ preserves the norm and therefore restricts, for every $\varepsilon>0$, to a diffeomorphism
$$
    \chi_\varepsilon\colon \partial D_\varepsilon(E_V)\longrightarrow \partial D_\varepsilon(E_W).
$$

Let $S(V)$ and $S(W)$ denote the unit spheres in $V$ and $W$. 
The free circle action on $S(V)\times S(W)$ given by $\lambda\cdot(v,w)=(\lambda v,\lambda^{-1}w)$ yields $U(V)\times U(W)$-equivariant diffeomorphisms
$$
    \partial D_\varepsilon(E_V) \xleftarrow[\cong]{\alpha_V} \frac{S(V)\times S(W)}{S^1} \xrightarrow[\cong]{\alpha_W}\partial D_\varepsilon(E_W),
$$
where $\alpha_V([v,w])=([v],\varepsilon v\otimes w)$ and $\alpha_W([v,w])=([w],\varepsilon v\otimes w)$.
Since $S(V)\cong S^3$ and $S(W)\cong S^5$, both boundaries are equivariantly diffeomorphic to $\frac{S^3\times S^5}{S^1}$. 
Moreover, $\chi_\varepsilon=\alpha_W\circ\alpha_V^{-1}$.
Thus, on the underlying smooth manifolds, the standard flip replaces $D_\varepsilon(E_V)$ by $D_\varepsilon(E_W)$ along this common boundary identification \cite[Section~3.1]{Fu-Hoskins-PLehalleur2023}.

Complex conjugation on $S(V)\times S(W)$ descends to the involution $[v,w]\mapsto[\overline v,\overline w]$ on the common quotient. 
It induces involutions
$$
    J_V(\ell,\xi)=(\overline{\ell},\overline{\xi}) \wideand J_W(m,\xi)=(\overline{m},\overline{\xi}).
$$
The definition of $\chi$ gives $\chi\circ J_V=J_W\circ\chi$ on $E_V\setminus\P(V)$. 
In particular,
\begin{equation}\label{eq:conjugation}
    \chi_\varepsilon\circ J_V=J_W\circ\chi_\varepsilon.
\end{equation}

Now let us prove our main theorem.
\begin{proof}[Proof of Theorem~\ref{thm:main}]
    Let $\Phi_0\colon B_{57}\to B_{50}$ be the orientation-preserving diffeomorphism obtained in Section~\ref{sec:fan-bott}.
    We will modify $\Phi_0$ by isotopies so that it becomes compatible with the local model of the flip. 

    First, we isotope $\Phi_0$ to a diffeomorphism $\Phi_1\colon B_{57}\to B_{50}$ that maps $C_{57}$ to~$C_{50}$ as complex conjugation. 
    For each $i\in\{50,57\}$, choose a tubular neighborhood embedding $\tau_i\colon D_{\varepsilon_0}(E_V)\to B_i$ compatible with the standard local model above, and let
    $$
        \iota_i\coloneq \left.\tau_i\right|_{\P(V)} \colon \CP^1\longrightarrow C_i\subset B_i.
    $$
    We choose these local identifications so that $\iota_i$ preserves the complex orientation. 
    Let $c\colon\CP^1\to\CP^1$ be complex conjugation.
    By \eqref{eq:homology-sign},
    $$
        (\Phi_0\circ\iota_{57})_*[\CP^1] = -[C_{50}]= (\iota_{50}\circ c)_*[\CP^1] \in H_2(B_{50};\Z).
    $$
    Since the Bott manifold $B_{50}$ is simply connected, the Hurewicz map $\pi_2(B_{50})\to H_2(B_{50};\Z)$ is an isomorphism. 
    Hence the embeddings~$\Phi_0\circ\iota_{57}$ and $\iota_{50}\circ c$ are homotopic.
    
    Since $\dim_\R B_{50}=8>2\dim_\R \CP^1+2$, the stable embedding theorem implies that these embeddings are isotopic~\cite[Proposition~II]{Mazur1963}. 
    By the isotopy extension theorem~\cite[Proposition~III]{Mazur1963}, this isotopy extends to an ambient isotopy $\Psi_t$ of $B_{50}$ with $\Psi_0 = \id_{B_{50}}$.
    Thus $\Phi_1 \coloneq \Psi_1 \circ \Phi_0$ is an orientation-preserving diffeomorphism isotopic to~$\Phi_0$ and satisfies $\Phi_1\circ\iota_{57}=\iota_{50}\circ c$.
    In particular, $\Phi_1(C_{57})=C_{50}$, and~$\Phi_1|_{C_{57}}=c$ in the chosen parametrizations. 
    
    Next, we show that the normal bundle map induced by $\Phi_1$ is homotopic to $J_V$, and then further isotope $\Phi_1$ to a diffeomorphism $\Phi\colon B_{57}\to B_{50}$ compatible with the local model of the flip.
    
    Identify both normal bundles with $E_V$ by means of the chosen tubular neighborhoods. 
    Under these identifications, let $A\colon E_V\to E_V$ be the normal bundle map induced by $d\Phi_1$. 
    Then $A$ is a real vector-bundle isomorphism covering $c$.
    Both $A$ and $J_V$ reverse the orientations of the real rank-six fibers. 
    Indeed, $A$ reverses the normal orientation because $\Phi_1$ preserves the ambient orientation while $c$ reverses the orientation of the center. 
    The map $J_V$ reverses the fiber orientation because complex conjugation on a complex three-dimensional vector space has orientation sign~$(-1)^3=-1$. 
    Consequently, $G\coloneq J_V^{-1}A$ is an orientation-preserving automorphism of the underlying real rank-six bundle covering $\id_{\CP^1}$.
    
    Choose a bundle metric on the underlying oriented real rank-six bundle.
    Fiberwise polar decomposition gives a homotopy from $G$ through orientation-preserving bundle automorphisms to an orthogonal bundle automorphism. 
    Such an automorphism is a section of the bundle of orientation-preserving orthogonal automorphisms of the fibers of $E_V$. 
    This bundle has fiber~$\SO(6)$ and a canonical identity section.
    Standard obstruction theory for sections \cite[Part~III]{Steenrod1999book} shows that the section determined by the orthogonal automorphism is homotopic to the identity section. 
    Indeed, the obstructions to a homotopy over~$S^2\times I$, relative to $S^2\times\partial I$, lie in
    $$
        H^{j+1}\bigl(S^2\times I,S^2\times\partial I;\pi_j(SO(6))\bigr)\cong H^j\bigl(S^2;\pi_j(SO(6))\bigr)
    $$
    for $j=0,1,2$. 
    These groups vanish because
    $$
    \pi_0(SO(6))=\pi_2(SO(6))=0 \wideand H^1(S^2;\pi_1(SO(6)))=0.
    $$
    Thus $G$ is homotopic to the identity through orientation-preserving bundle automorphisms.
    It follows that $A=J_VG$ is homotopic to $J_V$ through real bundle isomorphisms covering $c$.

    Fix such a homotopy $A_t\colon E_V\to E_V$ for $0\leq t\leq 1$ with $A_0=A$ and $A_1=J_V$.
    Since the fiberwise operator norms of $A_t$ are uniformly bounded on the compact space $\CP^1\times[0,1]$, one can choose $0<\varepsilon\leq\varepsilon_0$ so that $A_t\bigl(D_\varepsilon(E_V)\bigr) \subset D_{\varepsilon_0}(E_V)$ for every $t\in[0,1]$.
    
    Set $h_0 \coloneq  \left.\Phi_1\circ\tau_{57}\right|_{D_\varepsilon(E_V)}$, $h_A \coloneq\left.\tau_{50}\circ A\right|_{D_\varepsilon(E_V)}$, and $h_1 \coloneq\left.\tau_{50}\circ J_V\right|_{D_\varepsilon(E_V)}$.
    The embeddings $h_0$ and $h_A$ agree on the zero section and induce the same normal bundle map $A$. 
    By the uniqueness theorem for tubular neighborhoods~\cite[Chapter~4, Theorem~5.3]{Hirsch1976book}, they are isotopic relative to the zero section, after decreasing $\varepsilon$ if necessary.
    The family $\tau_{50}\circ A_t|_{D_\varepsilon(E_V)}$ is an isotopy from $h_A$ to $h_1$, also relative to the zero section.
    Thus,$h_0$ and $h_1$ are isotopic relative to the zero section.

    After concatenating these isotopies, the ambient tubular neighborhood theorem~\cite[Chapter~8, Theorem~1.8]{Hirsch1976book} gives an ambient isotopy~$\Theta_t$ of~$B_{50}$, supported in an arbitrarily small neighborhood of $C_{50}$, such that, after decreasing $\varepsilon$ further if necessary, $\Theta_1\circ\Phi_1\circ\tau_{57}=\tau_{50}\circ J_V$ on~$D_\varepsilon(E_V)$.
    Define $\Phi \coloneq \Theta_1\circ\Phi_1$.
    Then $\Phi$ is an orientation-preserving diffeomorphism isotopic to~$\Phi_0$, and
    \begin{equation}\label{eq:J_V}
        \Phi\circ\tau_{57} =\tau_{50}\circ J_V \quad\text{on }D_\varepsilon(E_V).
    \end{equation}

    Perform the flip on both $B_{57}$ and $B_{50}$ by replacing the chosen copies of~$D_\varepsilon(E_V)$ with~$D_\varepsilon(E_W)$. 
    This gives $X_{57}$ and $X_{50}$. 
    Define $F$ by using $\Phi$ on the complements of the removed disk bundles and $J_W$ on the inserted copies of~$D_\varepsilon(E_W)$.
    
    To verify the compatibility on the boundary, let $q\in\partial D_\varepsilon(E_V)$. 
    In $X_{57}$, the point~$\tau_{57}(q)$ on the boundary of the complement is identified with $\chi_\varepsilon(q)$ on the inserted disk bundle. By
    \eqref{eq:conjugation} and \eqref{eq:J_V},
    $$
        J_W\bigl(\chi_\varepsilon(q)\bigr)=\chi_\varepsilon\bigl(J_V(q)\bigr)=\chi_\varepsilon\bigl( \tau_{50}^{-1}\circ\Phi\circ\tau_{57}(q) \bigr).
    $$
    The right-hand side is precisely the image of $\Phi(\tau_{57}(q))$ under the boundary identification defining~$X_{50}$. 
    Hence the two definitions of $F$ agree on the common boundary.
    The identity $\chi\circ J_V=J_W\circ\chi$ and \eqref{eq:J_V} show that the two maps agree on a collar of the common boundary. 
    Hence they glue smoothly to a diffeomorphism~$F\colon X_{57}\to X_{50}$.

    Finally, since complex conjugation in complex dimension~$d$ has real orientation sign~$(-1)^d$ and the total space of $E_W$ has complex dimension four, $J_W$ preserves orientation. 
    Since $\Phi$ is also orientation-preserving, the glued diffeomorphism $F$ preserves orientation.
\end{proof}
The map $F$ constructed in Theorem~\ref{thm:main} is illustrated in Figure~\ref{fig:flip-surgery}.

\begin{figure}[t]
    \centering
    \begin{tikzpicture}[
        >=Stealth,
        every node/.style={font=\small},
        exterior/.style={draw, thick, rounded corners=12pt,fill=gray!12},
        oldpiece/.style={draw, thick, fill=gray!35},
        newpiece/.style={draw, thick, fill=gray!65},
        empty/.style={draw, thick, dashed, fill=white},
        roadmap/.style={draw, thick, rounded corners=4pt, fill=white, align=center, inner sep=4pt, font=\scriptsize}
    ]

    \draw[exterior] (0,1.5) rectangle (2.5,3.1);
    \draw[oldpiece] (1.25,2.3) ellipse [x radius=.65,y radius=.42];
    \node at (1.25,2.3) {$D_\varepsilon(E_V)$};
    \node at (1.25,3.3) {$B_i$};
    \draw[->,thick] (2.7,2.3) -- (3.6,2.3) node[midway,above] {remove};
    \draw[exterior] (3.8,1.5) rectangle (6.3,3.1);
    \draw[empty] (5.05,2.3) ellipse [x radius=.65,y radius=.42];
    \node[font=\scriptsize] at (5.05,2.3) {$\partial D_\varepsilon(E_V)$};
    \node at (5.05,3.3) {$B_i\setminus\operatorname{Int}D_\varepsilon(E_V)$};
    \draw[oldpiece] (4.25,.8) ellipse [x radius=.7,y radius=.4];
    \node at (4.25,.8) {$D_\varepsilon(E_V)$};
    \draw[->,thick] (4.8,1.9) to[out=225,in=80,looseness=1.05] (4.35,1.2);
    \draw[newpiece] (5.85,.8) ellipse [x radius=.7,y radius=.4];
    \node at (5.85,.8) {$D_\varepsilon(E_W)$};
    \draw[->,thick] (5.75,1.18) to[out=100,in=315,looseness=1.05] (5.3,1.9);
    \draw[->,thick] (6.5,2.3) -- (7.5,2.3) node[midway,above] {glue by $\chi_\varepsilon$};
    \draw[exterior] (7.7,1.5) rectangle (10.2,3.1);
    \draw[newpiece] (8.95,2.3) ellipse [x radius=.65,y radius=.42];
    \node at (8.95,2.3) {$D_\varepsilon(E_W)$};
    \node at (8.95,3.3) {$X_i$};

    \begin{scope}
        \clip[rounded corners=8pt] (0,-1.7) rectangle (9.4,-.7);
        \fill[gray!12] (0,-1.7) rectangle (9.4,-.7);
        \fill[gray!65] (6.6,-1.7) rectangle (9.4,-.7);
    \end{scope}

    \draw[thick,rounded corners=8pt] (0,-1.7) rectangle (9.4,-.7);
    \draw[thick,dashed] (6.6,-1.7) -- (6.6,-.7);
    \node at (3.3,-1.2) {$B_{57}\setminus \Int D_\varepsilon(E_V)$};
    \node at (8,-1.2) {$D_\varepsilon(E_W)$};
    \draw[thick] (0,-.55) -- (0,-.45) -- (9.4,-.45) -- (9.4,-.55);
    \node at (4.7,-.2) {$X_{57}$};

    \begin{scope}
        \clip[rounded corners=8pt] (0,-4.1) rectangle (9.4,-3.1);
        \fill[gray!12] (0,-4.1) rectangle (9.4,-3.1);
        \fill[gray!65] (6.6,-4.1) rectangle (9.4,-3.1);
    \end{scope}

    \draw[thick,rounded corners=8pt] (0,-4.1) rectangle (9.4,-3.1);
    \draw[thick,dashed] (6.6,-4.1) -- (6.6,-3.1);
    \node at (3.3,-3.6) {$B_{50}\setminus \Int D_\varepsilon(E_V)$};
    \node at (8,-3.6) {$D_\varepsilon(E_W)$};
    \node[roadmap] at (2.8,-2.4) { $\Phi_0 \xrightarrow{\;\text{match the center}\;} \Phi_1$ \\[-1pt] $\Phi_1 \xrightarrow{\;\text{match the normal data}\;} \Phi$};
    \draw[->,thick] (5.9,-1.85) -- (5.9,-2.95) node[midway,right] {$\Phi$};
    \draw[->,thick] (9,-1.85) -- (9,-2.95) node[midway,right] {$J_W$};
    \draw[->,thick] (-.45,-1.2) -- (-.45,-3.6) node[midway,left] {$F$};
    \draw[densely dashed,gray!70] (6.6,-1.7) -- (6.6,-3.1);
    \node[roadmap] at (7.75,-2.4) { $\chi_\varepsilon\circ J_V$\\[1pt]$=J_W\circ\chi_\varepsilon$ };

    \draw[thick] (0,-4.25) -- (0,-4.35) -- (9.4,-4.35) -- (9.4,-4.25);
    \node at (4.7,-4.6) {$X_{50}$};
    \end{tikzpicture}

    \caption{
        The flip replaces $D_\varepsilon(E_V)$ with $D_\varepsilon(E_W)$.
        Starting from the diffeomorphism~$\Phi_0\colon B_{57}\to B_{50}$, we obtain, through ambient isotopies, a diffeomorphism~$\Phi$ that agrees with the standard conjugation model on the tubular neighborhoods.
        The maps $\Phi$ on the complements and $J_W$ on the inserted disk bundles glue to form a diffeomorphism~$F\colon X_{57}\to X_{50}$.
    }
    \label{fig:flip-surgery}
\end{figure}
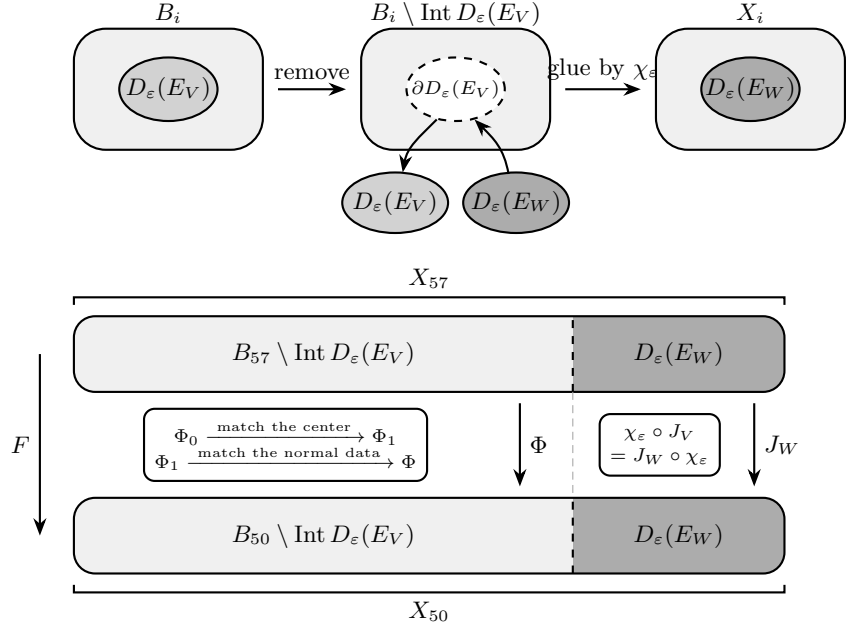

\section{Existence of diffeomorphisms across bistellar moves}

The proof of Theorem~\ref{thm:main} uses only the local model of a full-dimensional $2$--$(n-1)$ bistellar move. 
We record the resulting criterion.

Fix $n\ge 4$. 
For each $i\in\{0,1\}$, let $\Sigma_i^-$ and $\Sigma_i^+$ be smooth complete $n$-dimensional fans, and let $\cR_i$ denote their common set of rays.
Suppose that there are disjoint subsets $P_i,Q_i\subset\cR_i$ with 
$$
    |P_i|=2 \wideand |Q_i|=n-1,
$$
such that the underlying simplicial complex of $\Sigma_i^+$ is obtained from that of $\Sigma_i^-$ by the bistellar move
$$
    \Delta_{Q_i}*\partial\Delta_{P_i} \rightsquigarrow \partial\Delta_{Q_i}*\Delta_{P_i}.
$$
Assume also that
$$
    \sum_{\rho\in P_i}v_\rho =\sum_{\rho\in Q_i}v_\rho,
$$
where $v_\rho$ denotes the primitive generator of the ray $\rho$.

Set $B_i:=X_{\Sigma_i^-}$, $X_i:=X_{\Sigma_i^+}$, and $C_i:=V_{\Sigma_i^-}(Q_i)\cong\CP^1$.

\begin{theorem}\label{thm:criterion_flip}
    Suppose that there are a diffeomorphism~$\Phi\colon B_1\to B_0$ and a sign~$\varepsilon\in\{\pm1\}$ such that $\deg(\Phi)=\varepsilon^n$ and $\Phi_*[C_1]=\varepsilon[C_0]$.
    Then $\Phi$ is isotopic to a diffeomorphism compatible with the two flips. 
    Consequently, there is a diffeomorphism~$F\colon X_1\to X_0$ such that $\deg(F)=\deg(\Phi)=\varepsilon^n$.
\end{theorem}
\begin{proof}
    The proof is the same as that of Theorem~\ref{thm:main}, using the identity or complex conjugation on the standard local flip model according as $\varepsilon=1$ or $\varepsilon=-1$. 
    The assumption $n\ge 4$ gives the required stable embedding range.
\end{proof}

\begin{remark}
    Let $\omega_i\in H^{2n}(B_i;\Z)$ be the generator determined by the complex orientation, and set
    $$
        \alpha_i \coloneq \PD_{B_i}[C_i] = \prod_{\rho\in Q_i}[D_\rho] \in H^{2n-2}(B_i;\Z).
    $$
    Writing $\psi=\Phi^*$, the two sign conditions in Theorem~\ref{thm:criterion_flip} are equivalent to
    $$
        \psi(\omega_0)=\varepsilon^n\omega_1 \wideand \psi(\alpha_0)=\varepsilon^{n+1}\alpha_1.
    $$
    Thus, whenever the realizability of $\psi$ by a diffeomorphism is known, the remaining hypotheses can be checked from the fan data and the cohomology rings.
\end{remark}

\providecommand{\bysame}{\leavevmode\hbox to3em{\hrulefill}\thinspace}
\providecommand{\MR}{\relax\ifhmode\unskip\space\fi MR }
\providecommand{\MRhref}[2]{%
  \href{http://www.ams.org/mathscinet-getitem?mr=#1}{#2}
}
\providecommand{\href}[2]{#2}

\end{document}